\documentclass[final,5p,times,twocolumn]{elsarticle}

\usepackage{amssymb}
\usepackage{dsfont}
\usepackage{fancyhdr}
\usepackage{graphicx}
\usepackage{float}
\usepackage{subfigure}
\usepackage{amsfonts}
\usepackage{lineno,hyperref}
\usepackage{amsmath,amssymb,amsthm}
\newtheorem{remark}{Remark}[section]
\newtheorem{lemma}{Lemma}[section]

\newtheorem{theorem}{Theorem}[section]
\newtheorem{proposition}{Proposition}[section]
\newtheorem{definition}{Defintion}[section]

\begin{document}
\begin{frontmatter}\pagestyle{fancy}\thispagestyle{empty}
\setcounter{page}{1}
\chead{1}
   \pagenumbering{arabic}

\title{Observability on lattice points for  heat equations and  applications \tnoteref{ff}}
\tnotetext[ff]{This work was partially supported  by the National Natural Science Foundation of China under grants 11771344 and
11701535.}

\author[mysecondaryaddress]{Ming Wang}
\ead{mwangcug@outlook.com}

\author[my3address]{Can Zhang}
\ead{zhangcansx@163.com}

\author[mymainaddress]{Liang Zhang\corref{mycorrespondingauthor}}
\cortext[mycorrespondingauthor]{Corresponding author}
\ead{thanleon@163.com}

\address[mysecondaryaddress]{School of Mathematics and Physics, China University of Geosciences, Wuhan 430074,  China}
\address[my3address]{School of Mathematics and Statistics, Wuhan University, Wuhan 430072, China}
\address[mymainaddress]{Department of Mathematics, Wuhan University of Technology, Wuhan 430070, China}


\begin{abstract}
Observability inequalities on lattice points are established for non-negative solutions of the heat equation with potentials in the whole space. As applications, some controllability results of heat equations are derived by the above-mentioned observability inequalities.
\end{abstract}


\end{frontmatter}


\section{Introduction}
\label{}
This is a continuous research  of \cite{Wang17, Wang18} on observability inequalities for the heat equation in $\mathbb{R}^d$ ($d\geq1$)
\begin{equation}\label{heat}
\begin{cases}
\partial_t u = \Delta u, \quad  (t,x)\in \mathbb{R}^+\times \mathbb{R}^d;\\
  u(0,x) = u_0(x)\in L^2(\mathbb{R}^d).
\end{cases}
\end{equation}
Recall that a measurable set $E\subset \mathbb{R}^d$ is called an \emph{observable set} if for every $t>0$, there exists a constant $C(d,t,E)>0$ so that when $u$ solves \eqref{heat},
$$
\int_{\mathbb{R}^d}|u(t,x)|^2\,\mathrm dx \leq C(d,t,E)\int_0^t\int_E |u(s,x)|^2\,\mathrm dx\,\mathrm ds.
$$
It was shown in \cite{Wang17} (see also \cite{EV18}) that, $E$ is an observable set if and only if $E$ is $\gamma$-thick at scale $L$ for some positive $\gamma,L$, namely,
$$
\left|E \bigcap (x+LQ)\right|\geq \gamma L^d\;\;\mbox{for each}\;\;x\in \mathbb{R}^d.
$$
Here $Q$ is a unit cube in $\mathbb{R}^d$. Clearly, for every $N>0$, $E_N:=\mathbb{Z}^d/N=\{n/N:  n\in \mathbb{Z}^d\}$ is of zero measure (in the sense of $d$-dimensional Lebesgue measure), and thus it is not an observable set.

It was also shown in \cite{Wang18} that, for every $\varepsilon\in (0,1)$ and $t>0$, there exists a large enough $N=N(t,\varepsilon)>0$ so that we can, up to an $\varepsilon$ error, recover the solution of \eqref{heat} at the time $t$ by observing the solution on the set $E_N$ at the same time. More precisely, it follows from Theorem 1.2 (i) of \cite{Wang18} that, for  every $(\varepsilon, t)\in (0,1)\times \mathbb{R}^+$, there exists a constant $C=C(d)>1$ so that, if $N\geq  \sqrt{\frac{1}{t}\ln \frac{C}{\varepsilon}}$, then each solution to \eqref{heat} satisfies:
\begin{align}\label{equ-ob-app-1}
 \int_{\mathbb{R}^d}|u(t,x)|^2\,\mathrm dx\leq  2N^{-d}\sum_{n\in \mathbb{Z}^d} \big|u(t,\frac{n}{N})\big|^2 +\varepsilon \int_{\mathbb{R}^d}|u_0(x)|^2\,\mathrm dx.
\end{align}

Then, the following two natural open questions are remained to study:
\begin{description}
  \item[(1)] In general, can we remove the $\varepsilon$-term on the right hand side of \eqref{equ-ob-app-1}?
  \item[(2)] If not, for what kind of initial data, the $\varepsilon$-term in  \eqref{equ-ob-app-1}  can be removed?
\end{description}

For the first question, since $E_N$ is not an observable set, it is natural to expect that $\varepsilon$-term can not be removed. Actually, we shall construct an explicit example to illustrate it.
For the second question,  we obtain some sufficient conditions, though it is too hard to give a complete
characteristic for such kind of initial data. In all, our answers to these two questions are summarized in the following theorem.
\begin{theorem}\label{thm-ob}
\begin{description}
  \item[(i)]  The $\varepsilon$-term in \eqref{equ-ob-app-1} can not be removed in general.
  \item[(ii)] Assume that $u_0\geq 0$ (or $\leq 0$). Then we have the following estimate for all solutions of \eqref{heat}
\begin{align}\label{equ-12-2}
\int_{\mathbb{R}^d}u^2(t,x)\,\mathrm dx \leq 36^d e^{\frac{2d}{t}}\sum_{n\in \mathbb{N}^d}u^2(t,n), \quad t>0.
\end{align}
\end{description}
\end{theorem}

Two remarks are given in order. First, the inequality \eqref{equ-12-2} also holds (with a different upper bound constant) if the integral points $\mathbb{N}^d$ is replaced by $E_N$, defined as before. Second, the proof of \eqref{equ-12-2} is essentially based on a careful analysis of the heat kernel $K(t,x,y)=(4\pi t)^{-d/2}e^{-\frac{|x-y|^2}{4t}}$. In particular, we only need a Gaussian type upper bound and a lower bound of the kernel.

As it is well known that for a large class of potentials $V(x)$, $\Delta+V(x)$ generates an analytic semigroup $e^{t(\Delta+V)}$ in $L^2(\mathbb{R}^d)$, and that the kernel of the semigroup $e^{t(\Delta+V)}$ satisfies a two-side Gaussian type estimate. Thus, it is natural to extend the estimate in (ii) of Theorem \ref{thm-ob} to heat equations with potentials.

To this end, we consider the heat equation with a potential
\begin{equation}\label{heat-p}
\begin{cases}
\partial_t u =(\Delta+V(x)) u, \quad  (t,x)\in \mathbb{R}^+\times \mathbb{R}^d;\\
     u(0,x) = u_0(x)\in L^2(\mathbb{R}^d).
\end{cases}
\end{equation}
Here $V: \mathbb{R}^d\rightarrow \mathbb{R}$ depends only on the spatial variable. To state our result, we need the uniformly local Lebesgue integral spaces $L^p_{U, loc}(\mathbb{R}^d), p\geq 1,$ which are Banach spaces endowed with norms
$$
\|f\|_{L^p_{U, loc}(\mathbb{R}^d)}:= \sup_{x\in \mathbb{R}^d} \left( \int_{|x-y|\leq 1} |f(y)|^p\,\mathrm dy \right)^{\frac{1}{p}}.
$$
Clearly, the usual Lebesgue space $L^p(\mathbb{R}^d)$ is continuous embedding into $L^p_{U, loc}(\mathbb{R}^d)$.
\begin{theorem}\label{thm-ob-p}
Let $V$ be a real-valued function belonging to $L^p_{U, loc}(\mathbb{R}^d)$ with $p> \max\{1, \frac{d}{2}\}$. Assume that $u_0\geq 0$ (or $\leq 0$). Then there exists a constant $C=C(d,V)>0$ so that the following estimate hold for all solutions of \eqref{heat-p}
$$
\int_{\mathbb{R}^d}u^2(t,x)\,\mathrm dx \leq  e^{C(1+t+\frac{1}{t})}\sum_{n\in \mathbb{N}^d}u^2(t,n), \quad t>0.
$$
\end{theorem}

\medskip

The paper is organized as follows. In Section 2, we prove Theorem \ref{thm-ob} and Theorem \ref{thm-ob-p}.
In Section 3, we give some applications of the observability inequality in (ii) of Theorem \ref{thm-ob} in Control Theory.

\section{Proofs of main results}
\setcounter{equation}{0}
In the sequel, for every $x\in \mathbb{R}^d$ and $r>0$, we use $Q_r(x)$ to denote the closed cube in $\mathbb{R}^d$ centered at $x$ with side length $r$; We denote by $A^c$ the complement set of $A$.

\begin{lemma}\label{lem-out}
For any $a>0$ and $y\in Q_4^c(0)$, we have
\begin{align}\label{equ-717-2}
\sup_{x\in Q_2(0)} e^{-a|x-y|^2}\leq  2^{(d-1)}e^{\frac{(d-1)a}{2}}\sum_{n\in Q_2(0)\bigcap \mathbb{N}^d} e^{-a|n-y|^2}.
\end{align}
\end{lemma}
\begin{proof}
In the case that $d=1$, the inequality \eqref{equ-717-2} holds obviously.  We next assume that $d\geq 2$.  Arbitrarily give $y\in Q_4^c(0)$. Since $e^{-a|x-y|^2}$ is a continuous function of $x$ in $Q_2(0)$, the maximum of $e^{-a|x-y|^2}$ can be obtained at some point $x^*=(x_1^*,x_2^*, \cdots, x_d^*)$. Note that
\begin{align*}
\max_{x\in Q_2(0)}e^{-a|x-y|^2} &= \max_{x\in Q_2(0)}\prod_{i=1}^de^{-a|x_i-y_i|^2}\\
 &=\prod_{i=1}^d\max_{-1\leq x_i\leq 1}e^{-a|x_i-y_i|^2}=\prod_{i=1}^de^{-a|x_i^*-y_i|^2},
\end{align*}
where $x_i^*$ takes the form
\begin{equation}\label{equ-717-4}
x_i^*=
\begin{cases}
1, \quad y_i\geq 1;\\
y_i, \quad -1<y_i<1;\\
-1, \quad y_i\leq -1.
\end{cases}
\end{equation}
We write $y=(y_1,y_2,\cdots,y_d)$ and divide the set $\{y_i, i=1,2,\cdots,d\}$ into two groups: $|y_i|\geq 1$ and $|y_i|<1$.
Since $y\in Q_4^c(0)$, there exists at least one of $\{y_i, i=1,2,\cdots,d\}$ such that $|y_i|\geq 2$.  Thus there are at most $(d-1)$ elements of $\{y_i, i=1,2,\cdots,d\}$ satisfying $|y_i|<1$.

Without loss of generality we can assume that for some $j\leq d-1$
\begin{equation}\label{equ-717-5}
\begin{cases}
|y_{i}|\geq 1, \quad  j+1\leq i\leq d\\
-1<y_i<1, \quad i\leq j.
\end{cases}
\end{equation}
Then it follows from \eqref{equ-717-4} and \eqref{equ-717-5} that
\begin{equation}\label{equ-717-6}
x_i^*=
\begin{cases}
1\textmd{ or }-1, \quad j+1\leq i\leq d,\\
y_i, \quad i\leq j.
\end{cases}
\end{equation}
Then we have
\begin{equation}\label{equ-717-7}
\sup_{x\in Q_2(0)}e^{-a|x-y|^2}=e^{-a|x^*-y|^2}=\prod_{j+1\leq i\leq d}e^{-a|x_i^*-y_i|^2}.
\end{equation}

On the other hand, using \eqref{equ-717-6}, we have
\begin{multline}\label{equ-717-8}
\sum_{\substack{n=(n_1,n_2,\cdots,n_d)\\n_i\in \{-1,0,1\}, i=1,2,\cdots, d}} e^{-a|n-y|^2}\\
\geq \prod_{j+1\leq i\leq d}e^{-a|x_i^*-y_i|^2} \sum_{\substack{n_i\in \{-1,0,1\},\\ i=1,2,\cdots, j}} \prod_{i\leq j}e^{-a|n_i-y_i|^2}.
\end{multline}
This is because every term on the right hand side of \eqref{equ-717-8} appears on the left, and every term on the left is non-negative.
Combining \eqref{equ-717-7} and \eqref{equ-717-8}, we find that
\begin{align*}
\sum_{n\in Q_2(0)\bigcap \mathbb{N}^d} e^{-a|n-y|^2} &= \sum_{\substack{n=(n_1,n_2,\cdots,n_d)\\n_i\in \{-1,0,1\}, i=1,2,\cdots, d}} e^{-a|n-y|^2}\geq \Theta \sup_{x\in Q_2(0)}e^{-a|x-y|^2},
\end{align*}
where
$$
\Theta= \sum_{\substack{n_i\in \{-1,0,1\},\\ i=1,2,\cdots, j\leq d-1}} \prod_{i\leq j\leq d-1}e^{-a|n_i-y_i|^2}.
$$
Thus \eqref{equ-717-2} holds if one can show  that $2^{(d-1)}e^{\frac{(d-1)a}{2}}\Theta \geq 1.$

In fact, if we write
\begin{equation}\label{equ-717-10}
\sum_{\substack{n_i\in \{-1,0,1\},\\ i=1,2,\cdots, j}} \prod_{i\leq j}e^{-a|n_i-y_i|^2} = \prod_{i\leq j}A_i,  
\end{equation}
with 
$$
A_i= \sum_{k \in \{-1,0,1\}} e^{-a|k-y_i|^2}, \quad -1<y_i<1,
$$
then for every $1\leq i\leq j$, by the definition of $A_i$ we have
\begin{equation}\label{equ-717-12}
A_i\geq
\begin{cases}
\sum_{k \in \{0,1\}} e^{-a|k-y_i|^2}, \quad 0\leq y_i<1;\\
\sum_{k \in \{-1,0\}} e^{-a|k-y_i|^2}, \quad -1<y_i<0.
\end{cases}
\end{equation}
Thanks to \eqref{equ-717-12}, for $0\leq y_i<1$ we have
\begin{align*}
A_i&\geq  e^{-a|1-y_i|^2}+ e^{-ay_i^2} \geq \frac{1}{2}e^{-\frac{a}{2}(|1-y_i|^2+y_i^2)}\\
&=\frac{1}{2}e^{a(y_i-y_i^2-\frac{1}{2})}\geq \frac{1}{2}e^{-\frac{a}{2}}.
\end{align*}
Similarly, for $-1< y_i\leq 0$ we also have
$$
A_i\geq  e^{-a|-1-y_i|^2}+ e^{-ay_i^2}\geq \frac{1}{2}e^{-\frac{a}{2}}.
$$
Thus, we always have
\begin{equation}\label{equ-717-13}
A_i\geq \frac{1}{2}e^{-\frac{a}{2}}, \quad \quad i\leq j.
\end{equation}
It follows from \eqref{equ-717-10} and \eqref{equ-717-13} that
$$
\sum_{\substack{n_i\in \{-1,0,1\}, \\
i=1,2,\cdots, j\leq d-1}} \prod_{i\leq j\leq d-1}e^{-a|n_i-y_i|^2}\geq 2^{-j}e^{-\frac{aj}{2}}\geq 2^{-(d-1)}e^{-\frac{(d-1)a}{2}}.
$$
Thus   $2^{(d-1)}e^{\frac{(d-1)a}{2}}\Theta \geq 1.$  This completes the proof.
\end{proof}

\begin{lemma}\label{lem-inside}
For any $a>0$ and $y\in Q_4(0)$, we have
$$
\sup_{x\in Q_2(0)} e^{-a|x-y|^2}\leq e^{4ad}\sum_{n\in Q_2(0)\bigcap \mathbb{N}^d} e^{-a|n-y|^2}.
$$
\end{lemma}
\begin{proof}
Arbitrarily fix $y\in Q_4(0)$.  Since $\sup_{x\in Q_2(0)} e^{-a|x-y|^2}\leq 1$, it suffices to show that
\begin{equation}\label{equ-717-13.5}
1\leq e^{4ad}\sum_{n\in Q_2(0)\bigcap \mathbb{N}^d} e^{-a|n-y|^2}.
\end{equation}
Since $|y|\leq 2\sqrt{d}$, we have $\sum_{n\in Q_2(0)\bigcap \mathbb{N}^d} e^{-a|n-y|^2} \geq e^{-a|y|^2} \geq e^{-4ad}.$
This gives \eqref{equ-717-13.5} and finishes the proof.
\end{proof}


\begin{lemma}\label{lem-around}
For any $a>0$ and $x,y\in \mathbb{R}^d$, we have
$$
\sup_{z\in Q_2(x)} e^{-a|z-y|^2}\leq  2^{d-1}e^{4ad}\sum_{n\in Q_2(x)\bigcap (\mathbb{N}^d+x)} e^{-a|n-y|^2}.
$$
\end{lemma}
\begin{proof}
As a direct corollary of Lemma \ref{lem-out} and Lemma \ref{lem-inside},  it holds that
\begin{equation*}\label{equ-717-15}
\sup_{z\in Q_2(0)} e^{-a|z-y|^2}\leq  2^{d-1}e^{4ad}\sum_{n\in Q_2(0)\bigcap \mathbb{N}^d} e^{-a|n-y|^2}, \quad y\in \mathbb{R}^d.
\end{equation*}
By changing the variable $y\mapsto x+y$, the desired conclusion follows at once.
\end{proof}

\begin{theorem}\label{thm-around}
Assume that $u_0\geq 0$. Then for all solutions of \eqref{heat} and all $t>0$ and $k\in \mathbb{N}^d$
\begin{align}\label{equ-11-25-1}
u(t,x)\leq 2^{d-1}e^{\frac{d}{t}}\sum_{n\in Q_2(k)\bigcap \mathbb{N}^d} u(t,n), \quad x\in Q_2(k).
\end{align}
\end{theorem}

\begin{proof}
Using the heat kernel, the solution of the heat equation can be written as
\begin{equation}\label{equ-717-16}
u(t,x)=\frac{1}{(4\pi t)^{n/2}}\int_{\mathbb{R}^d}e^{-\frac{|x-y|^2}{4t}}u_0(y)\,\mathrm dy, \quad x\in \mathbb{R}^d.
\end{equation}
Applying Lemma \ref{lem-around} with $a=\frac{1}{4t}$, we obtain
\begin{equation}\label{equ-717-17}
e^{-\frac{|x-y|^2}{4t}} \leq 2^{d-1}e^{\frac{d}{t}}\sum_{n\in Q_2(k)\bigcap \mathbb{N}^d} e^{-\frac{|n-y|^2}{4t}}, \quad x\in Q_2(k).
\end{equation}
Since $u_0\geq 0$, the inequality \eqref{equ-11-25-1} follows from \eqref{equ-717-16} and \eqref{equ-717-17}.
\end{proof}

\begin{remark}\label{rem-1}
Notice that Theorem \ref{thm-around} does not follow from the parabolic Harnack inequality. The classical Harnack inequality \cite{Evans} says that, for every $t'>t>0, k\in \mathbb{N}^d$, every non-negative solution of \eqref{heat} satisfies that
\begin{align}\label{equ-11-28-0}
\max_{x\in Q_2(k)}u(t,x)\leq C(d,t,t')\inf_{x\in Q_2(k)} u(t',x).
\end{align}
The condition $t'>t$ is essential here. The time $t'$ cannot be equal to $t$ in \eqref{equ-11-28-0}. To see this, without loss of generality, it suffices to construct a non-negative solution such that the following fails:
\begin{align}\label{equ-11-27-1}
u(t,x_0)\leq 2^{d-1}e^{\frac{d}{t}}u(t,0),
\end{align}
where $x_0=(1,0,\cdots,0)\in \mathbb{R}^d$.

To this end, for every $M>0$, set $u_{0M}(x)=(4\pi)^{-\frac{d-1}{2}}\chi_{\{M\leq x_1\leq M+1\}}e^{-\frac{|x'|^2}{4}}$, where $\chi_{\{M\leq x_1\leq M+1\}}$ is the characteristic function of the interval $[M,M+1]$, $x=(x_1,x')\in \mathbb{R}^d$. Clearly, $u_{0M}$ is uniformly bounded in $L^2(\mathbb{R}^d)$. Using the heat kernel, we find the solution of the heat equation \eqref{heat} with initial datum $u_{0M}$ is given by
\begin{equation*}\label{equ-11-28-1}
u_M(t,x) = (4\pi(t+1))^{-\frac{d-1}{2}}e^{-\frac{|x'|^2}{4(t+1)}} (4\pi t)^{-1/2}\int_M^{M+1} e^{-\frac{|x_1-y_1|^2}{4t}}\,\mathrm dy_1, 
\end{equation*}
for all $t>0, x\in \mathbb{R}^d$.
By some computations, we have for $t>0$
\begin{align}\label{equ-11-28-2}
u_M(t,x_0) &= (4\pi(t+1))^{-\frac{d-1}{2}} (4\pi t)^{-\frac{1}{2}}\int_M^{M+1} e^{-\frac{|1-y_1|^2}{4t}}\,\mathrm dy_1\nonumber\\
&\geq\frac{1}{2}(4\pi(t+1))^{-\frac{d-1}{2}}(4\pi t)^{-\frac{1}{2}}e^{-\frac{|\frac{1}{2}-M|^2}{4t}},
\end{align}
\begin{align}\label{equ-11-28-3}
u_M(t,0) &= (4\pi(t+1))^{-\frac{d-1}{2}} (4\pi t)^{-\frac{1}{2}}\int_M^{M+1} e^{-\frac{|y_1|^2}{4t}}\,\mathrm dy_1\nonumber\\
&\leq (4\pi(t+1))^{-\frac{d-1}{2}} (4\pi t)^{-\frac{1}{2}}e^{-\frac{M^2}{4t}}.
\end{align}
Combining \eqref{equ-11-28-2} and \eqref{equ-11-28-3} gives that
\begin{align}\label{equ-11-28-4}
u_M(t,x_0) \geq \frac{1}{2}u_M(t,0)e^{\frac{M-\frac{1}{4}}{4t}},   \quad t>0.
\end{align}
When $M$ is large enough, \eqref{equ-11-28-4} obviously contradicts with \eqref{equ-11-27-1}.
\end{remark}

\begin{proof}[\textbf{Proof of (ii) of Theorem \ref{thm-ob}.}]
Without loss of generality, we can assume that $u_0\geq 0$. Arbitrarily fix $k\in \mathbb{N}^d$.
The number of points of the set $Q_2(k)\bigcap \mathbb{N}^d$ is $3^d$. Using the elementary inequality
$\left(\sum_{1\leq i\leq m}a_i\right)^2\leq m\sum_{1\leq i\leq m}a_i^2$ with $m=3^d$, we deduce from \eqref{equ-11-25-1} that
\begin{align}\label{equ-718-1}
u^2(t,x)\leq {12}^de^{\frac{2d}{t}}\sum_{n\in Q_2(k)\bigcap \mathbb{N}^d} u^2(t,n), \quad x\in Q_2(k).
\end{align}
Integrating \eqref{equ-718-1} on $x\in Q_2(k)$, noting that the volume of $Q_2(k)$ is $2^d$, we get
\begin{align}\label{equ-718-2}
\int_{Q_2(k)}u^2(t,x)\,\mathrm dx\leq {24}^de^{\frac{2d}{t}}\sum_{n\in Q_2(k)\bigcap \mathbb{N}^d} u^2(t,n).
\end{align}
Finally, summarizing \eqref{equ-718-2} for $k\in \mathbb{N}^d$  we deduce that
\begin{multline*}
\int_{\mathbb{R}^d}u^2(t,x)\,\mathrm dx = 2^{-d} \sum_{k\in \mathbb{N}^d} \int_{Q_2(k)}u^2(t,x)\,\mathrm dx\\
\leq {12}^de^{\frac{2d}{t}}\sum_{k\in \mathbb{N}^d}\sum_{n\in Q_2(k)\bigcap \mathbb{N}^d} u^2(t,n) \leq 36^d e^{\frac{2d}{t}}\sum_{n\in \mathbb{N}^d}u^2(t,n).
\end{multline*}
This proves the theorem.
\end{proof}

\begin{proof}[\textbf{Proof of (i) of Theorem \ref{thm-ob}.} ]For every $T>0$ and $N>0$, it suffices to show that  there exists an initial datum $u_0\in L^2(\mathbb{R}^d)$ such that the following inequality fails
\begin{align}\label{equ-eps}
\int_{\mathbb{R}^d}|u(T,x)|^2\,\mathrm dx\leq  N^{-d}\sum_{n\in \mathbb{Z}^d} \big|u(T,\frac{n}{N})\big|^2.
\end{align}

We first consider  the case that $d=1$. Define an initial datum $u_0$ via Fourier transform\footnote{Here the Fourier transform is defined as $\mathcal {F}f = \widehat{f}(\xi) = \int_{\mathbb{R}^d}e^{-ix\cdot\xi}f(x)\,\mathrm dx$, and the inverse Fourier transform is $\mathcal {F}^{-1}f (x)= (2\pi)^{-d}\int_{\mathbb{R}^d}e^{ix\cdot\xi}f(\xi)\,\mathrm d \xi$.}
$$
\widehat{u_0}(\xi) = \frac{\pi}{i} e^{T\xi^2}\Big(\widehat{f}(\xi-N\pi)-\widehat{f}(\xi+N\pi)\Big),
$$
where $\widehat{f}(\xi)=e^{-(T+1)\xi^2}$, $\xi\in \mathbb{R}$. Clearly, 
$$
e^{T\xi^2}\widehat{f}(\xi-N\pi)=e^{-\xi^2+2(T+1)N\pi\xi-(T+1)(N\pi)^2}
$$
belongs to $L^2(\mathbb{R})$. So does $e^{T\xi^2}\widehat{f}(\xi+N\pi)$. Thus, $\|u_0\|_{L^2(\mathbb{R})}=\|\widehat{u_0}\|_{L^2(\mathbb{R})}<\infty.$ Since $u(T,x)=(e^{T\triangle}u_0)(x)$, we find
$$
\widehat{u(T,\cdot)}(\xi)=e^{-T\xi^2}\widehat{u_0}(\xi)=\frac{\pi}{i}\Big(\widehat{f}(\xi-N\pi)-\widehat{f}(\xi+N\pi)\Big).
$$
Taking the inverse Fourier transform we obtain
$$
u(T,x)=\frac{1}{2i}(e^{iN\pi x }-e^{-iN\pi x })f(x)=\sin (N\pi x) f(x).
$$
Since $f$ is a bounded smooth function, we find $u(T,\frac{n}{N})=0$ for $n\in \mathbb{Z}$. However, it is clear that $\|u(T,\cdot)\|_{L^2(\mathbb R)}>0$. This leads to a contradiction with \eqref{equ-eps} in one dimension.

In higher dimensions, set
$$
\widehat{u_0}(\xi) = \frac{(2\pi)^d}{2i} e^{T|\xi|^2}\prod_{i=1}^d\Big(\widehat{f}(\xi_i-N\pi)-\widehat{f}(\xi_i+N\pi)\Big)
$$
with $\widehat{f}(\xi_i)=e^{-(T+1)\xi_i^2}$, $\xi_i\in \mathbb{R}$. Similar to the analysis above, we find
$$
u(T,\frac{n}{N})=0, n\in \mathbb{Z}^d,    \quad \textmd{but }  \|u(T,\cdot)\|_{L^2(\mathbb R^d)}>0.
$$
This completes the proof.
\end{proof}

\begin{definition}
We say that a function $V(\cdot)$ in $L_{loc}^1(\mathbb{R}^d)$ satisfies two-side Gaussian type heat kernel estimates, if the operator $\Delta+V$ generates an analytic semigroup $e^{t(\Delta+V(x))}$ in $L^2(\mathbb{R}^d)$,  and if there exist positive constants
$c_i$ $(i=1,2,3,4)$, with  $c_2\leq c_4$, so that for all $t>0, x,y\in \mathbb{R}^d$
$$
t^{-\frac{d}{2}}e^{-c_1(t+1)}e^{-\frac{|x-y|^2}{c_2t}}\leq K(t,x,y)\leq t^{-\frac{d}{2}}e^{c_3(t+1)}e^{-\frac{|x-y|^2}{c_4t}}, 
$$
where $K(t,x,y)$ is the kernel of the semigroup $e^{t(\Delta+V)}$, namely,
$$
(e^{t(\Delta+V)} f)(x) = \int_{\mathbb{R}^d} K(t,x,y)f(y)\,\mathrm dy.
$$
\end{definition}

\begin{theorem}\label{thm-2}
Let $V$ be a real-valued function satisfying two-side Gaussian type heat kernel estimates. Assume that $u_0\geq 0$ (or $\leq 0$). Then there exists a positive constant $C=C(d,V)$ such that the following estimate hold for all solutions of \eqref{heat-p}
$$
\int_{\mathbb{R}^d}u^2(t,x)\,\mathrm dx \leq  e^{C(1+t+\frac{1}{t})}\sum_{n\in \mathbb{N}^d}u^2(t,n), \quad t>0.
$$
\end{theorem}
\begin{proof} We only consider the case that $u_0\geq 0$.
Since both $V$ and  $u_0$ are real-valued, the solution $u$ of \eqref{heat-p} is also real-valued. According to the definition of the kernel $K(t,x,y)$, we have
\begin{equation*}\label{equ-11-29-1}
u(t,x) = \int_{\mathbb{R}^d} K(t,x,y)u_0(y)\,\mathrm dy.
\end{equation*}
Since $V$ satisfies two-side Gaussian type heat kernel estimates, we find that for all $t>0$ and $x\in \mathbb{R}^d$
\begin{align}\label{equ-11-29-2}
0\leq u(t,x) &\leq \int_{\mathbb{R}^d} t^{-\frac{d}{2}}e^{c_3(t+1)}e^{-\frac{|x-y|^2}{c_4t}}u_0(y)\,\mathrm dy\nonumber\\
&=(c_4\pi)^{\frac{d}{2}}e^{c_3(t+1)} (e^{\frac{c_4}{4}t\Delta}u_0)(x).
\end{align}
We apply Theorem \ref{thm-ob} with $t$ replaced by $\frac{c_4}{4}t$ to obtain that
\begin{align}\label{equ-11-29-3}
\int_{\mathbb{R}^d}((e^{\frac{c_4}{4}t\Delta}u_0)(x))^2\,\mathrm dx \leq 36^d e^{\frac{8d}{c_4t}}\sum_{n\in \mathbb{N}^d}((e^{\frac{c_4}{4}t\Delta}u_0)(n))^2.
\end{align}
Combining \eqref{equ-11-29-2} and \eqref{equ-11-29-3} gives that for all $t>0$
\begin{multline}\label{equ-11-29-4}
\int_{\mathbb{R}^d}u^2(t,x)\,\mathrm dx\\ \leq 36^d(c_4\pi)^de^{2c_3(t+1)} e^{\frac{8d}{c_4t}}\sum_{n\in \mathbb{N}^d}((e^{\frac{c_4}{4}t\Delta}u_0)(n))^2.
\end{multline}
Replacing $t$ by $c_2t/c_4$ in \eqref{equ-11-29-4} gives that for all $t>0$
\begin{multline}\label{equ-11-29-4.5}
\int_{\mathbb{R}^d}u^2(\frac{c_2}{c_4}t,x)\,\mathrm dx \\\leq 36^d(c_4\pi)^de^{2c_3(\frac{c_2}{c_4}t+1)} e^{\frac{8d}{c_2t}}\sum_{n\in \mathbb{N}^d}((e^{\frac{c_2}{4}t\Delta}u_0)(n))^2.
\end{multline}

On the other hand, using the lower bound of the kernel, we find that
\begin{align}\label{equ-11-29-5}
u(t,x) &\geq \int_{\mathbb{R}^d} t^{-\frac{d}{2}}e^{-c_1(t+1)}e^{-\frac{|x-y|^2}{c_2t}}u_0(y)\,\mathrm dy\nonumber\\
&=(c_2\pi)^{\frac{d}{2}}e^{-c_1(t+1)} (e^{\frac{c_2}{4}t\Delta}u_0)(x).
\end{align}
It follows from \eqref{equ-11-29-5} that for all $t>0$ and $n\in \mathbb{N}^d$
\begin{align}\label{equ-11-29-6}
((e^{\frac{c_2}{4}t\Delta}u_0)(n))^2 \leq (c_2\pi)^{-d}e^{2c_1(t+1)}u^2(t,n).
\end{align}
Inserting \eqref{equ-11-29-6} into \eqref{equ-11-29-4.5} we get
\begin{align}\label{equ-11-29-6.5}
\int_{\mathbb{R}^d}u^2(\frac{c_2}{c_4}t,x)\,\mathrm dx \leq (\frac{36 c_4}{c_2})^de^{2c_3(\frac{c_2}{c_4}t+1)}e^{2c_1(t+1)} e^{\frac{8d}{c_2t}}\sum_{n\in \mathbb{N}^d}u^2(t,n).
\end{align}
Moreover, using the upper bound of $K$ again, we find that for all $t>0$
\begin{align}\label{equ-11-29-7}
\|e^{t(\Delta+V)}\|_{L^2(\mathbb{R}^d), L^2(\mathbb{R}^d)} &\leq \|t^{-\frac{d}{2}}e^{c_3(t+1)}e^{-\frac{|x|^2}{c_4t}}\|_{L^1(\mathbb{R}^d)}\nonumber\\
&=(c_4\pi)^{d/2}e^{c_3(t+1)}.
\end{align}
Noting that $c_2/c_4\leq 1$, it follows from \eqref{equ-11-29-7} (since $u(t,x)=(e^{t(\Delta+V)}u_0)(x)$) that
\begin{align}\label{equ-11-29-8}
\int_{\mathbb{R}^d}u^2(t,x)\,\mathrm dx \leq (c_4\pi)^{d}e^{2c_3(2-\frac{c_2}{c_4})}\int_{\mathbb{R}^d}u^2(\frac{c_2}{c_4}t,x)\,\mathrm dx.
\end{align}
Finally, combining \eqref{equ-11-29-6.5} and \eqref{equ-11-29-8} gives the desired conclusion.
\end{proof}

\begin{proof}[\textbf{Proof of Theorem \ref{thm-ob-p}.}]
Since $V\in L^p_{U, loc}(\mathbb{R}^d)$ with $p> \max\{1, \frac{d}{2}\}$, it is easy to check (see also \cite[Proposition 2.1]{Zheng}) that $V$ belongs to the Kato class $K_d$ (see \cite[p. 453]{Simon} for a precise definition). According to \cite[Theorem 7.1]{Simon}, for all $\varepsilon>0$, there exist positive constants $C_1(\varepsilon), C_2(\varepsilon)$ so that the kernel $K(t,x,y)$ of the analytic semigroup $e^{t(\Delta+V)}$ satisfies
$$
t^{-\frac{d}{2}}e^{-C_1(\varepsilon)(1+t)}e^{-\frac{|x-y|^2}{4(1-\varepsilon)t}}\leq K(t,x,y)\leq t^{-\frac{d}{2}}e^{C_2(\varepsilon)(1+t)}e^{-\frac{|x-y|^2}{4(1+\varepsilon)t}}
$$
for all $t>0, x,y\in \mathbb{R}^d.$
Thus $V$ satisfies a two-side Gaussian type heat kernel estimate. Then Theorem \ref{thm-ob-p} follows from Theorem \ref{thm-2} directly.
\end{proof}

\section{Applications to controllability}
\setcounter{equation}{0}
In this section, we will show an application of Theorem \ref{thm-ob} for an impulsive controllability for the heat equation in $\mathbb R^d$.  We refer the interesting reader to \cite{EV18,Wang17} for the null controllability of the heat equation in $\mathbb R^d$
with distributed controls.

Arbitrarily fix $T>0$ and $\tau\in(0,T)$.
Consider the heat equation with impulsive control
\begin{equation}\label{well}
\left\{
\begin{split}
&\partial_t y(t,x)-\Delta y(t,x)=0\;\;&\text{in}\;\;(0,T)\times\mathbb R^d,\\
&y(\tau,x)=y(\tau^-,x)+Bv\;\;&\text{in}\;\;\mathbb R^d,\\
&y(0)=y_0(x)\;\;&\text{in}\;\;\mathbb R^d.
\end{split}
\right.
\end{equation}
Here, $y$ is the state variable, $y_0\in L^2(\mathbb R^d)$, $y(\tau^-,x)$ denotes the left limit of $y(\cdot,x)$ (treated as
a function from $\mathbb R^+$ to $\mathbb R^d$) for each $x$ at time $\tau$,
and $v\in  \ell^2(\mathbb R^d)$ is the control.
The control operator $B: \ell^2(\mathbb R^d)\rightarrow \mathcal D'(\mathbb R^d)$ is defined by
\begin{equation*}
Bv:=\sum_{n\in\mathbb Z^d}v_n\delta_n\;\;\text{for each}\;\;v=(v_n)_{n\in\mathbb Z^d}\in \ell^2(\mathbb R^d),
\end{equation*}
where $\delta_n(x):=\delta_n(x-n)$, $x\in\mathbb R^d$, with $\delta$ being the Dirac measure.
In fact, it is not hard to check that $B$ is  linear and bounded from $\ell^2(\mathbb R^d)$ to $H^{-s}(\mathbb R^d)$
when $s>d/2$ (see also \cite{Wang18}).

We first quote from \cite{Wang18} the following result concerning the well-posedness  of \eqref{well}.
\begin{proposition}
If $s>d/2$,   $y_0\in L^2(\mathbb R^d)$ and  $v\in  \ell^2(\mathbb R^d)$, then \eqref{well} has a unique
solution in $C([0,\tau)\cup(\tau,T];L^2(\mathbb R^d))\cap C([\tau,T];H^{-s}(\mathbb R^d))$.
\end{proposition}

The main result of this section is stated  as follows.
\begin{theorem}\label{signcon}
Let $0<\tau<T$. For each $y_0\in L^2(\mathbb R^d)$, there is a control $v\in \ell^2(\mathbb R^d)$, with
$$
\|v\|_{\ell^2}\leq Ce^{\frac{C}{T-\tau}}\|y_0\|,
$$
such that the solution of \eqref{well} verifies $y(T,x;v)\geq 0$ for a.e. $x\in \mathbb R^d$.
\end{theorem}

Here and in the sequel, we write $\langle\cdot,\cdot\rangle$ and $\|\cdot\|$ for the usual inner product and norm
in $L^2(\mathbb R^d)$, and denote by $\langle\cdot,\cdot\rangle_{\ell^2}$ and $\|\cdot\|_{\ell^2}$ the usual inner product and norm
in $\ell^2(\mathbb R^d)$, respectively.

\begin{remark}
Similarly, for each $y_0\in L^2(\mathbb R^d)$, there exists $v\in \ell^2(\mathbb R^d)$ such that the solution of \eqref{well} satisfies $y(T,x;v)\leq 0$ for a.e. $x\in \mathbb R^d$.
\end{remark}

\begin{remark}
Analogous results can be established for the heat equation with potentials by using Theorem \ref{thm-ob-p} instead.
\end{remark}

We point out that the proof of Theorem \ref{signcon} is motivated and  adapted from the arguments in \cite[Theorem 2.4]{Y} (see also \cite{LG}).
Before giving the proof of Theorem \ref{signcon}, we start with some preliminaries.
First of all, we quote from \cite[Chapter I, Theorem 1.2]{ms} the following classical result in the Calculus of Variations.

\begin{proposition}\label{cv}
Let $X$ be a reflexive Banach space and let $K$ be a weakly closed subset. If a weakly lower semi-continuous
functional $F:K\rightarrow\mathbb R$ satisfies the following coercive condition
$$
\lim_{\substack{z\in K\\ \|z\|_X\rightarrow +\infty}}F(z)=+\infty,
$$
then $F$ attains its minimum in $K$, i.e., there exists at least one $\hat z$ so that $$F(\hat z)=\min_{z\in K}F(z).$$
Moreover, if $F$ is strictly convex in $K$ then it has a unique minimizer.
\end{proposition}
\begin{remark}
Notice that closed and convex subsets of Banach spaces are important examples of weakly closed sets.
\end{remark}

In the sequel, we define
\begin{equation*}
L^2_+(\mathbb R^d):=\big\{y\in L^2(\mathbb R^d): y(x)\geq0 \;\;\text{for a.e.}\;\;x\in \mathbb R^d\big\}.
\end{equation*}
Clearly, it is a closed and convex subset of $L^2(\mathbb R^d)$.

For each $T>\tau>0$ and $\varepsilon>0$, we define a functional $F^{T,\tau}_\varepsilon: L^2_+(\mathbb R^d)\rightarrow \mathbb R$ by
\begin{equation*}
F^{T,\tau}_\varepsilon(\varphi_T):=\frac{1}{2}\|B^*\varphi(\tau)\|^2_{\ell^2}+\langle y_0,\varphi(0)\rangle
+\varepsilon \|\varphi_T\|
\end{equation*}
for any $\varphi_T\in L^2_+(\mathbb R^d)$, where $y_0$ is the given initial state of \eqref{well} and $\varphi$ solves the so-called adjoint equation
\begin{equation}\label{adj}
\left\{
\begin{split}
&\partial_t \varphi(t,x)+\Delta \varphi(t,x)=0\;\;&\text{in}\;\;(0,T)\times\mathbb R^d,\\
&\varphi(T,x)=\varphi_T(x)\;\;&\text{in}\;\;\mathbb R^d,
\end{split}
\right.
\end{equation}
and $B^*:C_0^\infty(\mathbb R^d)\rightarrow \ell^2(\mathbb R^d)$ is the adjoint operator of $B$.
It is clear that $B^*$ is linear and bounded from $H^{s}(\mathbb R^d)$ to $\ell^2(\mathbb R^d)$ with $s>d/2$.

\begin{lemma}\label{du6}
Given $T>\tau>0$ and $y_0\in L^2(\mathbb R^d)$, for each $\varepsilon>0$, $F^{T,\tau}_\varepsilon$ has a unique minimizer, denoted by $\widehat \varphi^\varepsilon_T$, in $L^2_+(\mathbb R^d)$. Furthermore, for all $\varphi_T\in L^2_+(\mathbb R^d)$
\begin{equation}\label{du1}
\langle B^*\widehat \varphi^\varepsilon(\tau), B^* \varphi(\tau)\rangle_{\ell^2}+\langle y_0,\varphi(0)\rangle+\varepsilon\|\varphi_T\|\geq0,
\end{equation}
where $\widehat \varphi^\varepsilon$ and $\varphi$ are the solutions of \eqref{adj} with  $\widehat \varphi_T^\varepsilon$ and $\varphi_T$, respectively.
\end{lemma}
\begin{proof}
It is not hard to check that $F^{T,\tau}_\varepsilon$ is strictly convex and weakly lower semi-continuous in $L^2_+(\mathbb R^d)$. We next show that  $F^{T,\tau}_\varepsilon$ satisfies the coercive condition, i.e.,
$$
\liminf_{\substack{\varphi_T\in L_+^2(\mathbb R^d)\\ \|\varphi_T\|\rightarrow +\infty}}F^{T,\tau}_\varepsilon(\varphi_T)=+\infty.
$$
To seek a contradiction, we would assume that there was a sequence  $\{\varphi_T^n\}_{n\geq1}\subset L_+^2(\mathbb R^d)$ to be such that
$$\lim_{n\rightarrow+\infty}\|\varphi_T^n\|=+\infty$$
and
$$F^{T,\tau}_\varepsilon(\varphi_T^n)<+\infty\;\;\text{for all}\;\;n\in \mathbb N. $$
Let us set
$$\widetilde{\varphi}_T^n:=\frac{\varphi_T^n}{\|\varphi_T^n\|}\;\;\text{for all}\;\;n\in \mathbb N. $$
Clearly, $\widetilde{\varphi}_T^n\in L^2_+(\mathbb R^d)$.
Then
$$
\frac{F^{T,\tau}_\varepsilon(\varphi_T^n)}{\|\varphi_T^n\|}
=\frac{1}{2}\|\varphi_T^n\|\|B^* \widetilde\varphi^n(\tau)
\|_{\ell^2}^2+\langle y_0,\widetilde\varphi^n(0)\rangle+\varepsilon,
$$
where $\widetilde \varphi^n$ is the solution to
\begin{equation*}
\left\{
\begin{split}
&\partial_t\widetilde \varphi^n+\Delta \widetilde \varphi^n=0\;\;&\text{in}\;\;(0,T)\times\mathbb R^d,\\
&\widetilde \varphi^n(T)=\widetilde \varphi_T^n\;\;&\text{in}\;\;\mathbb R^d.
\end{split}
\right.
\end{equation*}
As $F^{T,\tau}_\varepsilon(\varphi_T^n)$ is uniformly bounded, we have
\begin{equation}\label{25can1}
\lim_{n\rightarrow +\infty}\|B^*\widetilde\varphi^n(\tau)
\|_{\ell^2}=0.
\end{equation}
Because
$$\|\widetilde\varphi_T^n\|=1\quad\text{for all}\;\;n\geq1,$$
there exists $\widetilde\varphi_T\in L^2_+(\mathbb R^d)$ and a subsequence $\{\widetilde\varphi_T^{n_k}\}_{k\geq1}$ such that
$$
\widetilde\varphi_T^{n_k}\rightharpoonup \widetilde\varphi_T\;\;\text{weakly in}\;\;L^2(\mathbb R^d).
$$
Hence,
\begin{equation}\label{11c1}
\widetilde\varphi^{n_k}(s)\rightarrow \widetilde\varphi(s)\;\;\text{in}\;\;L^2(\mathbb R^d)\;\;\text{for each}\;\;s\in[0,T),\end{equation}
where
$\widetilde\varphi$ is the solution to
\begin{equation*}
\left\{
\begin{split}
&\partial_t\widetilde\varphi+\Delta \widetilde\varphi=0\;\;&\text{in}\;\;(0,T)\times\mathbb R^d,\\
&\widetilde\varphi(T)=\widetilde\varphi_T\;\;&\text{in}\;\;\mathbb R^d.
\end{split}
\right.
\end{equation*}
By \eqref{25can1} and \eqref{11c1}, it holds that
$$\|B^*\widetilde\varphi(\tau)\|_{\ell^2}=0.$$
This, along with (ii) in Theorem \ref{thm-ob}, implies that $\widetilde\varphi\equiv0$ in $[0,T]\times\mathbb R^d$.
Consequently,
$$
\liminf_{k\rightarrow+\infty}\frac{F^{T,\tau}_\varepsilon(\varphi_T^{n_k})}{\|\varphi_T^{n_k}\|}
\geq \varepsilon.
$$
This leads to a contradiction with the uniform boundedness of  $F^{T,\tau}_\varepsilon(\varphi_T^{n_k})$ for all $k\geq1$. Hence, the first part of this lemma follows from Proposition \ref{cv} immediately.

For the second part of this lemma, we first note that
\begin{equation}\label{du2}
\lim_{\rho\rightarrow 0}\frac{F^{T,\tau}_\varepsilon(\widehat \varphi_T^\varepsilon+\rho\varphi_T)-F^{T,\tau}_\varepsilon(\widehat \varphi_T^\varepsilon)}{\rho}\geq 0
\end{equation}
for all $\rho>0$ and $\varphi_T\in L^2_+(\mathbb R^d).$
If $\widehat \varphi_T^\varepsilon=0$, then \eqref{du1} is obviously valid by \eqref{du2}.
Otherwise, by the definition of $F^{T,\tau}_\varepsilon$, one can easily derive that
\begin{multline}\label{du3}
\lim_{\rho\rightarrow 0}\frac{F^{T,\tau}_\varepsilon(\widehat \varphi_T^\varepsilon+\rho\varphi_T)-F^{T,\tau}_\varepsilon(\widehat \varphi_T^\varepsilon)}{\rho}\\
= \langle B^*\widehat \varphi^\varepsilon(\tau), B^* \varphi(\tau)\rangle_{\ell^2}+\langle y_0,\varphi(0)\rangle+\varepsilon\langle \frac{\widehat \varphi_T^\varepsilon}{\|\widehat \varphi_T^\varepsilon\|},\varphi_T\rangle
\end{multline}
for any $\varphi_T\in L^2_+(\mathbb R^d)$. Noting  that
$$
\langle \frac{\widehat \varphi_T^\varepsilon}{\|\widehat \varphi_T^\varepsilon\|},\varphi_T\rangle\leq
\|\varphi_T\|,
$$
therefore, \eqref{du1} follows from \eqref{du2} and \eqref{du3}. It completes the proof.
\end{proof}

\begin{lemma}\label{du7}
For each $\varepsilon>0$,
let $\widehat \varphi_T^\varepsilon$ be the minimizer of $F^{T,\tau}_\varepsilon$ in $L^2_+(\mathbb R^d)$, and let $\widehat \varphi^\varepsilon$ be the solution of \eqref{adj} with  $\widehat \varphi_T^\varepsilon$.
Then there exists a positive constant $C$, independent of $\varepsilon$, so that
\begin{equation}\label{du4}
\|B^*\widehat \varphi^\varepsilon(\tau)\|_{\ell^2}\leq Ce^{\frac{C}{T-\tau}}\|y_0\|\quad \text{for all}\quad \varepsilon>0.
\end{equation}
\end{lemma}

\begin{proof}
Since $F^{T,\tau}_\varepsilon(\widehat \varphi_T^\varepsilon)\leq F^{T,\tau}_\varepsilon(0)=0$, we see that for any $\varepsilon>0$
\begin{equation}\label{du5}
\|B^*\widehat \varphi^\varepsilon(\tau)\|_{\ell^2}^2\leq 2\|y_0\|\|\widehat \varphi^\varepsilon(0)\|\leq 2\|y_0\|\|\widehat \varphi^\varepsilon(\tau)\|.
\end{equation}
Thanks to (ii) of Theorem \ref{thm-ob},  we have
\begin{equation*}
\|\widehat \varphi^\varepsilon(\tau)\| \leq Ce^{\frac{C}{T-\tau}}\|B^*\widehat\varphi^\varepsilon(\tau)\|_{\ell^2} \quad
\text{for all}\;\; \varepsilon>0,
\end{equation*}
with a positive constant $C$ independent of $\varepsilon$. This, together with \eqref{du5}, implies \eqref{du4} at once.
\end{proof}

\begin{proof}[\textbf{Proof of Theorem \ref{signcon}}]
Arbitrarily  fix $\varphi_T\in L_+^2(\mathbb R^d)$. Let $\varphi$ be any solution of \eqref{adj} with $\varphi(T)=\varphi_T$.  Multiplying the equation \eqref{well} by $\varphi$
and then integrating the resulting  over $[0,T]\times\mathbb R^d$, we obtain that for any $v\in \ell^2(\mathbb R^d)$
\begin{equation}\label{di}
\langle y(T;v),\varphi_T\rangle-\langle y_0,\varphi(0)\rangle=\langle v,B^*\varphi(\tau)\rangle_{\ell^2} .
\end{equation}

Then, for any $\varepsilon>0$,
we let $\widehat \varphi_T^\varepsilon$ be the minimizer of $F^{T,\tau}_\varepsilon$ and let $\widehat \varphi^\varepsilon$ be the solution of \eqref{adj} with $\varphi(T)=\widehat \varphi_T^\varepsilon$.  Now, by setting
$$v_\varepsilon:=B^*\widehat\varphi_\varepsilon(\tau),\;\;\varepsilon>0,$$
 in the above identity \eqref{di},  we get
\begin{equation*}
\langle y(T;v_\varepsilon),\varphi_T\rangle-\langle y_0,\varphi(0)\rangle=\langle B^*\widehat\varphi_\varepsilon(\tau),B^*\varphi(\tau)\rangle_{\ell^2}.
\end{equation*}
This, combined with Lemma \ref{du6}, indicates that
\begin{equation}\label{du8}
\langle y(T;v_\varepsilon),\varphi_T\rangle+\varepsilon\|\varphi_T\|\geq0\quad \text{for any}\quad \varepsilon>0.
\end{equation}

Finally, by Lemma \ref{du7}, we see that $v_\varepsilon$ is uniformly bounded in $\ell^2(\mathbb R^d)$, and thus there exists $\hat v\in \ell^2(\mathbb R^d)$ satisfying
$$\|\hat v\|_{\ell^2}\leq Ce^{\frac{C}{T-\tau}}\|y_0\|,$$
where the constant $C>0$ is independent of $\varepsilon$, such that (up to a subsequence)
$$
v_\varepsilon \rightharpoonup \hat v \quad\text{weakly in}\quad \ell^2(\mathbb R^d)\quad\text{as}\;\;\varepsilon\rightarrow 0.
$$
Hence, by letting $\varepsilon$ goes to zero in \eqref{du8}, we at once obtain that
$$
\langle y(T;\hat v),\varphi_T\rangle \geq0.
$$
This completes the proof because of the arbitrariness of $\varphi_T$ in $L^2_+(\mathbb R^d)$.
\end{proof}


\end{document}